\documentclass[12pt]{extarticle}
\usepackage{amsfonts,amsmath}
\usepackage{amssymb}
\usepackage{latexsym}
\usepackage{epsfig}
\usepackage{graphicx,color,graphics}
\usepackage{color}
\newtheorem{lemma}{Lemma}[section]

\newtheorem{theorem}[lemma]{Theorem}

\newcommand{\N}{\ifmmode{{\Bbb N}}\else{\mbox{${\Bbb N}$}}\fi}
\newcommand{\R}{\ifmmode{{\Bbb R}}\else{\mbox{${\Bbb R}$}}\fi}

\begin{document}

%+Title
\title{Exponential stability of the Euler-Bernoulli microbeam and thermal effect. }

\author{Roberto D\'{i}az\footnote{Departamento de Ciencias Exactas, Universidad de los Lagos, Osorno, Chile, (roberto.diaz@ulagos.cl)}\ , \ Octavio Vera\footnote{Departamento de Matem\'{a}tica, Universidad
del B\'{i}o-B\'{i}o, Concepci\'{o}n, Chile, (octaviovera49@gmail.com)} \ and \ Nicol\'{a}s Zumelzu \footnote{Departamento de F\'{i}sica y Matem\'{a}tica, Universidad de Magallanes, Punta Arenas,
Chile, (nicolas.zumelzu@umag.cl)}}

\date{}
\maketitle \pagestyle{plain} \thispagestyle{plain}
%-Title

%+Abstract
\begin{abstract}
The main goal in this work is to prove the exponential decay of the semigroup associated with a thermoelastic system composed of an Euler-Bernoulli type equation that models the transverse oscillation of a homogeneous microbeam with axial movement and in which a viscous damping is acting. In addition, to this microbeam has been endowed with a thermal effect given by the Coleman-Gurtin model which depends essentially on past history, representing an improvement of Fourier, Cattaneo and Green-Naghdi models. To achieve this goal, we will use mainly multiplicative techniques and standard tools of functional analysis.
\end{abstract}
%-Abstract
\renewcommand{\theequation}{\thesection.\arabic{equation}}
\setcounter{equation}{0}
\section{Introduction}\label{Sec1}

\hspace{0.4cm} The use of electromechanical systems of small dimensions (dimensions of the order of micrometers or nanometers) has received in the last time an integral treatment within the scientific literature, this is due to its multiple and interesting applications in science and technology. On the other hand, due to the small dimensions of some of the pieces that make up these systems (or parts of them), they must ideally be modeled as microbeams. One of the main aspects to be modeled as beams or microbeams devices of these tiny systems, is to take into account the scale effects of the material and that these allow an analytical and mechanical treatment of high degree of significance for both engineering and industry. An important factor that can be highlighted from theories that take into account a minimum particle size, is that they can be applied without restrictions to structures of natural scale, when the materials that compose it are of high homogeneity.

Regarding the analysis of the oscillations in the beams, this has played a fundamental role in the technical literature of recent years. Namely one of the main problems of both the physical and mathematical point of view, is the question of the stabilization of the vibrations of a structure and as is known there are several types of stability but the most important of all is exponential stability. On the other hand, if a structure is given a heat effect, a sufficiently precise and realistic model is created from the physical point of view \cite{car,ame,riv}. At present moment, the effects of heat in a structure are mainly modeled by two non-classical modern perspectives that, at the limit case, involve the classical Fourier law of heat conduction, which has as its main deficiency the infinite propagation of heat in a body \cite{nun}, but there are several models that improve this deficiency such as the Cattaneo's law or Green-Naghdi type III model, however non of them consider the effects of past history in the heat flow (see \cite{g,fava,ch} ), unlike the model Coleman-Gurtin, which is essentially based on the assumption that the heat pulse is influenced by the past history of the temperature gradient and eliminate the paradoxes and controversies that the previous models \cite{g}.

On the other hand, according to Newton's laws, the constitutive equation that models the transverse oscillations of a microbeam with constant axial movement and in which an external disturbing force acting on it, is given by the following relation \cite{chen,din}

\begin{equation}\label{a}
M_{xx}+\rho\,A\left( u_{tt}+2\,\kappa\, u_{xt}+\kappa^{2}\,u_{xx} \right)=F(x,t)
\end{equation}where $u(x,t)$ is the transverse displacement, $\rho$ is the material density of the medium, $\kappa$ is the axial speed acting on the beam, $M(x,t)$ is the bending moment, $A$ is cross-sectional area and $F(x,t)$  is the external force acting on it. According to the theory of beams of Euler-Bernoulli and when applied to the model \eqref{a} the equation of the transverse movement of the beam with axial movement $\kappa$ can be deduced and which is given by the equation:
\begin{equation}\label{a0}
u_{xxxx}+\frac{\rho\,A}{E\,I}\left( u_{tt}+2\,\kappa\, u_{xt}+\kappa^{2}\,u_{xx} \right)+\alpha_{T}\,\left[ M_{T}\right]_{xx}=\frac{1}{E\,I}\,F(x,t)
\end{equation}
In this case $E\,I$ is the flexural rigidity of the
beam, $M_{T}$ is the thermal moment, $\alpha_{T}$ is
the thermal expansion coefficient. \\

\noindent
It can be highlighted firstly that in the study of this type of problems it is ideally sought that the semigroup associated with these equations or systems coupled with some dissipative effect decay exponentially when time $t$ tends to infinity. In this sense we note that in general there are several contributions to the study of the asymptotic behavior of systems associated with thermoviscoelastic problems with memory which in essence physically allow resistance to the change of state of the system, namely \cite{gio,gio2,fen,pata,nun,jo}. In this direction the following works can be enunciated when the heat flow model has been modified or follows a non-classical heat law, whose results are closely related to the one presented to this paper:\\

\noindent
Abouregal et al.  \cite{a}  they study the following system provided with an initial condition of the sinusoidal type.
\begin{align}
\begin{aligned}
&u_{xxxx}+\frac{\rho\,A}{E\,I}\left( u_{tt}+2\,\kappa\, u_{xt}+\kappa^{2}\,u_{xx} \right)+\alpha_{T}\,\left[ M_{T}\right]_{xx}=\frac{1}{E\,I}\,F(x,t) \\
& K\,\left( \theta_{xx}+\theta_{zz}\right)=\left( 1+\tau_{0}\frac{\partial}{\partial t}\right)\left[\rho \,C_{E}\,\theta_{t}-\gamma T_{0}\,z\, u_{txx} \right] \label{a1}
\end{aligned}
\end{align}
and they show through the tools of the numerical analysis that under certain conditions of the variable $\theta$ and the axial constant $\kappa$ of the previous system, this decays exponentially as time increase. \\

\noindent
 D\'iaz et al. in \cite{rd} study a variation of the previous problem and using standard tools of semigroups of linear operators theory, prove the well-posedness of the problem and the exponential stability of the following system
\begin{align}
\begin{aligned}
& u_{tt} + \left(\, p(x)\,u_{xx}\,\right) _{xx} + 2\,g(x)\,u_{t} + 2\,\kappa\,u_{xt} -\kappa^{2} \,u_{xx} +\beta\,\theta_{x} = 0 \\
& \theta_{tt}+\theta_{t}-\kappa\theta_{xx}-\eta u_{xxt}-\xi\theta_{xxt}=0 \label{a11}
\end{aligned}
\end{align}

In this work we will study the equation \eqref{a11}$_{1}$ with a thermal effect acting on it, such a thermal effect is given by the model of Coleman-Gurtin \cite{col} , which considers the effects of past history and gives us a more efficient and realistic mathematical model from the physical point of view. Namely our system is composed of the following equations
\begin{align}
\label{101}& u_{tt} + \left(\, \textbf{p}(x)\,u_{xx}\,\right) _{xx} + 2\,\textbf{g}(x)\,u_{t} + 2\,\kappa\,u_{xt} -\kappa^{2} \,u_{xx} +\beta\,\theta_{x} = 0,\quad\text{in}\quad\varGamma
\\
& \theta_{t}-\frac{1-\lambda_{1}}{\lambda_{2}}\,\theta_{xx}-\frac{\lambda_{1}}{\lambda_{2}}\,\int^{\infty}_{0}\varphi(s)\,\theta_{xx}(t-s) \ ds+\beta\,u_{xt}=0, \quad\text{in}\quad\varGamma.\label{102}
\end{align}
Where $\varGamma=\Omega\times\mathbb{R}^{+}=(\,0\,,L\,)\times (\,0\,,+\infty\,)$  , $\textbf{p}(x)\in H^{2}(\Omega)$ which describes the non-homogeneity of the material, $\textbf{g}(x)\in L^{\infty}(\Omega)$ represents the force viscous damping \cite{herr}.  The constants $\beta\,,\kappa\,,\lambda_{2}$ are assumed to be strictly positive and  $\lambda_{1}\in(\,0\,,1\,)$ , the memory kernel $\varphi$ is considered to be regular and positive function tend to zero as time increases.\\

For the problem \eqref{101}--\eqref{102}, we consider the following boundary values
\begin{equation}
\begin{aligned}
\label{103}u(0,\,t)=u(L,\,t)=u_{x}(0,  t)=u_{x}(L,\,t)=0,\quad \theta(0,\,t)=\theta(L,\,t)=0\quad \forall t\geq 0
\end{aligned}
\end{equation}
and initial conditions
\begin{align}
\label{104} u(x,\,0)=u_{0}(x),\quad u_{t}(x,\,0)=u_{1}(x),\quad
\theta(x,\,t)\rvert_{t\leq 0}=\theta_{0}(x),\quad \forall x\in\Omega
\end{align}
Throughout this paper, we always assume that there exist real constants $\alpha_{i}\,(\,i=1,...,4\,)$  such that for $x\in[0\,,L]$ and that satisfy the following estimates.
\begin{align}
\begin{aligned}	
\label{105} & 0 <\alpha_{1}\leq \textbf{p}(x)\leq
\alpha_{2}, \\
& 0 <\alpha_{3}\leq \textbf{g}(x)\leq \alpha_{4}.
\end{aligned}	
\end{align}
\\
The main goals of this work are to establish the well-posedness of the problem \eqref{101} - \eqref{104} considering the smoothness assumptions of the functions $\varphi, \textbf{p}, \textbf{g} $ and prove that the energy of the system decays exponentially as time $t$ tends to infinity.\\

\noindent
This paper is organized as follows: Section 2, we consider the system \eqref{101}--\eqref{104} and we give some notations and contents needed to deploy our main results . In Section 3, well-posedness of the system is established. In Section 4, we will prove that energy of the system under consideration decay to zero when $t$ tends to infinity.

\section{Functional setting and notation}
\label{sec:1}
In this section, we provide the semigroup context and the main tools wich be used to obtain the main result  . We consider the Lebesgue and Sobolev spaces

\begin{equation*}
L^{p}(\Omega),\quad 1\leq p \leq \infty,\hspace{0,3cm} H_{0}^{1}(\Omega)\hspace{0,3cm} \text{and}\hspace{0,3cm} H_{0}^{2}(\Omega)
\end{equation*}
in the case $p=2$ we write $\lVert u \rVert$ instead of $\lVert u \rVert_{2}$ .

In order to write the system \eqref{101}--\eqref{104} as a Cauchy problem in a Hilbert space, we introduce the new variable in the form proposed by \cite{gio,fen}
\begin{equation}\label{106}
\eta^{t}(x,s)=\eta(s)=\int^{s}_{0}\theta(x,\,t-\tau)\ d\tau,\quad (x\,,s)\in\Omega\times\mathbb{R}^{+}\,,t\geq 0
\end{equation}
Note that \eqref{106} satisfies the equation
\begin{equation*}
\eta_{t}(x,s)+\eta_{s}(x,s)=\theta(x,t) ,\quad(x\,,s)\in\Omega\times\mathbb{R}^{+},\,t\geq 0
\end{equation*}
and
\begin{equation*}
\eta^{t}(0)=0, \quad \text{in} \quad \mathbb{R}^{+},\,t\geq 0
\end{equation*}
also
\begin{equation*}
\eta^{0}(x,s)=\int^{s}_{0}\theta(x,-\tau)\ d \tau, \quad  (x\,,s)\in\Omega\times\mathbb{R}^{+}
\end{equation*}

Now substituting in the original system the variables $(u\,,v\,,\theta\,,\eta)$ where $v=u_{t}$ and we consider $\mu(s)=-\frac{\lambda_{1}}{\lambda_{2}}\,\varphi'(s)\,, l=\frac{1-\lambda_{1}}{\lambda_{2}}>0$ we see that \eqref{101}--\eqref{104} be yields the following equivalent system
\begin{align}
& v_{t}+\left( \,\textbf{p}(x)\,u_{xx}\,\right)_{xx}+2\,\textbf{g}(x)\,v+2\,\kappa\, v_{x}-\kappa^{2}\,u_{xx}+\beta\,\theta_{x}=0, \quad\text{in}\quad\varGamma \label{107}\\
& \theta_{t}-l\,\theta_{xx}-\int^{\infty}_{0}\mu(s)\,\eta_{xx}(s) \ ds+\beta\,v_{x}=0, \quad\text{in}\quad\varGamma \label{108}\\
& \eta_{t}+\eta_{s}=\theta. \quad\text{in}\quad\varGamma\times\mathbb{R}^{+}.\label{109}
\end{align}
Thus, the boundary conditions become
\begin{align}
\begin{aligned}	
&\label{110}u(0\,,t)=u(L\,,t)=u_{x}(0,  t)=u_{x}(L\,,t)=0,\quad \theta(0\,,t)=\theta(L\,,t)=0,\\
& \eta(0\,,s)=\eta(L\,,s)=0, \quad t\geq 0\,, s>0.
\end{aligned}	
\end{align}

\noindent
and the initial conditions are given by
\begin{align}
\begin{aligned}	
& u(x\,,0)=u_{0}(x),\quad u_{t}(x\,,0)=u_{1}(x)=v_{0}(x),\quad
\theta(x\,,0)=\theta_{0}(x)\\
& \eta^{0}(x\,,s)=\eta_{0}(x\,,s),\quad  (x\,,s)\in\Omega\times\mathbb{R}^{+}. \label{111}
\end{aligned}
\end{align}
\\
The kernel within \eqref{108} satisfies the following hypothesis
\begin{align}
&\mathbb{H}_{1}:\mu(s)\in C^{1}\left( \mathbb{R}^{+}\right)\cap L^{1}\left( \mathbb{R}^{+}\right),\hspace{0,1cm}\mu(s)\geq 0, \label{b}\\
& \mathbb{H}_{2}: \mu'(s) \leq 0 \hspace{0,2cm} \text{on}\hspace{0,2cm}\mathbb{R}^{+},\label{c}\\
& \mathbb{H}_{3}:\mu_{0}=\int^{\infty}_{0}\mu(s) \ ds > 0, \label{d}\\
& \mathbb{H}_{4}: \exists \, \delta_{1}>0, \hspace{0,1cm} \text{such that }\hspace{0,1cm}
\mu'(s)+\delta_{1}\,\mu(s) \leq 0.\label{f}
\end{align}

In view of  $\mathbb{H}_{1}-\mathbb{H}_{4}$,  let's assume that the $\eta$ variable belongs to the weighted Sobolev space with given by:
\begin{align*}
{\cal M}=L_{\mu}^{2}\left( \mathbb{R}^{+}\,, H_{0}^{1}(\Omega)\right)=\left\lbrace\, \eta:\mathbb{R}^{+}\rightarrow H_{0}^{1}(\Omega):\int_{0}^{\infty}\mu(s)\,\Arrowvert\eta_{x}(s)\Arrowvert^{2}<\infty \right\rbrace \\
\end{align*}
and the norm of this space is given :
\begin{equation*}
\left\|\eta\right\|_{{\cal M}}^{2} =\int_{0}^{\infty}\mu(s)\Arrowvert\eta_{x}\Arrowvert^{2}\ ds
\end{equation*}
Now, we define the following phase space
\begin{equation*}
\mathcal{H}= H_{0}^{2}\times
L^{2}\times
L^{2}\times
{\cal M}
\end{equation*}
The inner product on ${\cal H}$ is given by
\begin{align}
\langle\, \Phi,\Phi_{1}\,\rangle_{\mathcal{ H}}  = &  \int_{\Omega}\textbf{p}(x)\,u_{xx}\,\overline{u}_{1xx}\ dx+ \kappa^{2}\int_{\Omega} u_{x}\,\overline{u}_{1x}\ dx+\int_{\Omega} v\,\overline{v}_{1}\ dx+\int_{\Omega}\theta\,\overline{\theta}_{1}\ dx  \nonumber \\
&+\int_{\Omega}\int^{\infty}_{0} \mu(s)\, \eta_{x}\,\overline{\eta}_{1x}\ ds \ dx \label{113}
\end{align}
where $\Phi(t)=\Phi=\left( u\,,v\,,\theta\,,\eta\right)^{\top} ,\ $
$\Phi_{1}(t)=\Phi_{1}=\left( u_{1}\,,v_{1}\,,\theta_{1}\,,\eta_{1}\right)^{\top} $.\\

\noindent
The space ${\cal H}$ is equipped with the induced norm
\begin{equation}\label{114}
\|\Phi\|_{\mathcal{H}}^{2}  = \left\| \sqrt{p(x)}\,u_{xx}\right\|^{2}+\kappa^{2}\,\left\| u_{x}\right\|^{2}+ \left\| v\right\|^{2}+\left\| \theta\right\|^{2}+\left\| \eta\right\|_{{\cal M}}^{2}
\end{equation}
\\
On the other hand the initial value problem \eqref{107}--\eqref{111} can be reduced to the following abstract initial value problem for a first order evolution equation
\begin{equation}\label{115}
\frac{d}{dt}\Phi(t)=\mathbb{A}\,\Phi(t),\qquad \Phi(0)=\Phi_{0}
\end{equation}
where
\begin{equation*}
\Phi_{0}=\left( u_{0}\,,v_{0}\,,\theta_{0}\,,\eta_{0}\right)^{\top}.
\end{equation*}
\\

\noindent
The linear operator   $\mathbb{A}:\textbf{D}(\mathbb{A})\subset
\mathcal{H}\rightarrow \mathcal{H}$ is given by
\begin{equation}\label{116}
\mathbb{A}\,\Phi =\left(
\begin{array}{c}
v \\
-\left( \, \textbf{p}(x)\,u_{xx}\,\right)_{xx} -2\,\textbf{g}(x)\,v-2\,\kappa\, v_{x}+\kappa^{2}\, u_{xx}-\beta\,\theta_{x} \\
l\,\theta_{xx}+\int^{\infty}_{0}\mu(s)\,\eta_{xx}(s) \ ds-\beta\,v_{x} \\
\theta-\eta_{s}
\end{array}
\right)
\end{equation}

\noindent
with domain
$\textbf{D}(\mathbb{A})$ of the operator $\mathbb{A}$ is defined by
\begin{align*}
\textbf{D}\left( \mathbb{A}\right) =\left\{\Phi\in {\cal H}\left|\begin{array}{l} \textbf{p}(x)\,u_{xx}-\kappa^{2}\,u\in H^{2}\,, v\in H_{0}^{2}\\
\theta\in H^{1}_{0}\,,\eta\in {\cal M},\,\eta_{s}\in {\cal M},\,\eta(0)=0\\
l\,\theta_{xx}+\int^{\infty}_{0}\mu(s)\,\eta_{xx}(s) \ ds \in L^{2}	
\end{array} \right. \right\}
\end{align*}
\section{Well posedness}
\label{sec:2}
\begin{theorem}
	The operator $\mathbb{A}$ generates a $C_{0}$-semigroup $ \textbf{B}(t)=e^{\mathbb{A}\,t} $ of contractions on the space $\mathcal{H}.$
\end{theorem}
\begin{proof}
	We will show that  $\mathbb{A}$ is a dissipative operator and the mapping $I-\mathbb{A}$ is surjective . Firstly we observe that $\overline{\textbf{D}(\mathbb{A})}={\cal H}$, by using  \eqref{113} and we have for any $\Phi\in \textbf{D}\left( \mathbb{A}\right)$ that:
	\begin{align}
	\langle\, \mathbb{A}\,\Phi,\,\Phi\,\rangle_{\mathcal{H}} \nonumber
	& =\int_{\Omega}\textbf{p}(x)\,v_{xx}\,\overline{u}_{xx} \ dx+\kappa^{2}\,\int_{\Omega}v_{x}\,\overline{u}_{x} \ dx +\int_{\Omega}\left[-\left( \, \textbf{p}(x)\,u_{xx}\,\right)_{xx} -2\,\textbf{g}(x)\,v  \right]\,\overline{v} \ dx\\
	& +\int_{\Omega}\left[ -2\,\kappa\, v_{x}+\kappa^{2}\, u_{xx}-\beta\,\theta_{x} \right]\,\overline{v} \ dx +\int_{\Omega}\left[l\,\theta_{xx}+ \int^{\infty}_{0}\mu(s)\,\eta_{xx}(s) \ ds-\beta\,v_{x}\right]\,\overline{\theta}\ dx \nonumber\\
	& +\int_{\Omega}\int^{\infty}_{0} \mu(s)\, \left(\theta-\eta_{s} \right)_{x} \,\overline{\eta}_{x}\ ds \ dx \label{117}
	\end{align}
	Integrating by parts in \eqref{117} and straightforward calculations we have
	\begin{align}
	\langle\, \mathbb{A}\,\Phi,\,\Phi\,\rangle_{\mathcal{H}} \nonumber
	& =2\,i\,Im\int_{\Omega}\textbf{p}(x)\,\overline{u}_{xx}\,v_{xx} \ dx+2\,i\,\kappa^{2}\,Im\int_{\Omega}\overline{u}_{x}\,v_{x} \ dx+2\,i\,\beta\,Im\int_{\Omega}\overline{v}_{x}\,\theta \ dx \nonumber\\
	&+2\,i\,Im\int_{\Omega}\int^{\infty}_{0} \mu(s)\, \theta_{x} \,\overline{\eta}_{x}\ ds \ dx-2 \int_{\Omega}\textbf{g}(x)\,|v|^{2}\ dx-l \int_{\Omega}|\theta_{x}|^{2}\ dx \nonumber\\ & +\frac{1}{2}\int_{0}^{+\infty}\mu'(s)\Arrowvert\eta_{x}\Arrowvert^{2}\ ds \label{117b}
	\end{align}
	Taking real parts in \eqref{117b} we obtain	
	\begin{align}
	Re\left\langle\,\mathbb{A}\,\Phi\,,\Phi
	\,\right\rangle_{\mathcal{H}}=-2 \int_{\Omega}\textbf{g}(x)\,|v|^{2}\ dx-l \int_{\Omega}|\theta_{x}|^{2}\ dx +\frac{1}{2}\int_{0}^{+\infty}\mu'(s)\Arrowvert\eta_{x}\Arrowvert^{2}\ ds\leq 0 \label{118}
	\end{align}
	thus $\mathbb{A}$ is a dissipative operator.
	\\
	
	On the other hand, we prove that the operator $I-\mathbb{A}:\textbf{D}\left( \mathbb{A}\right)\rightarrow {\mathcal H}$ is surjective, we will use similar ideas given in \cite{gior1,fen}. Let any $\Phi^{*}=\left( u^{*}\,,v^{*}\,,\theta^{*}\,,\eta^{*}\right)^{\top}\in{\cal{H}}$, and consider the equation
	\begin{equation*}
	\left( I-\mathbb{A}\right)\,\Phi=\Phi^{*}
	\end{equation*}
	which can be written as
	\begin{align}
	& u-v =u^{*},\label{119}\\
	& v+\left( \textbf{p}(x)\,u_{xx}\right)_{xx} +2\,\textbf{g}(x)\,v+2\,\kappa\, v_{x}-\kappa^{2}\,u_{xx}+\beta\,\theta_{x}=v^{*}\label{120}\\
	& \theta-l\,\theta_{xx}-\int^{\infty}_{0}\mu(s)\,\eta_{xx}(s) \ ds+\beta\,v_{x}=\theta^{*}\label{121}\\
	& \eta-\theta+\eta_{s}=\eta^{*}\label{122}
	\end{align}
	Integrating  \eqref{122} we obtain
	\begin{equation}\label{123}
	\eta(s)=\left( 1-e^{-s}\right)\theta+ \int_{0}^{s}e^{\tau-s}\,\eta^{*}\ d\tau
	\end{equation}
	Replacing \eqref{123} in \eqref{121} and \eqref{119} in \eqref{120}--\eqref{121} and considering that
	\begin{equation}\label{124}
	\vartheta=\int^{\infty}_{0}\mu(s)\int^{s}_{0}e^{\tau-s}\,\eta_{xx}^{*}\ d\tau \ ds
	\end{equation}
	belongs to $H^{-1}$ (see \cite{gio2}) the systems   \eqref{119}-- \eqref{122} becomes
	\begin{align}
	&\textbf{c}_{0}\,u+2\,\kappa\,u_{x}+\left( \textbf{p}(x)\,u_{xx}\right)_{xx}-\kappa^{2}\,u_{xx}+\beta\,\theta_{x}=\textbf{c}_{0}\,u^{*}+2\,\kappa\,u_{x}^{*}+v^{*}\label{125}  \\
	&\theta-(l+\textbf{c}_{1})\,\theta_{xx}+\beta\,u_{x}=\beta\,u_{x}^{*}+\theta^{*}+\vartheta \label{126}
	\end{align}
	where
	\begin{align*}
	& \textbf{c}_{0}=1+2\,\textbf{g}(x) \\
	& \textbf{c}_{1}=\int_{0}^{\infty}\mu(s)\,\left( 1-e^{-s}\right)\ ds
	\end{align*}
	namely the function $ \textbf{c}_{0} $ is bounded and positive by \eqref{105} and the constant $\textbf{c}_{1} $ is positive under $\mathbb{H}_{1}-\mathbb{H}_{3}$. Now consider the bilinear form $ {\mathcal J} \left (z,z \right) $ associated with \eqref{125} - \eqref {126} and which is given by:
	
	\begin{equation*}\label{127}
	{\mathcal J}(z\,,z)=\left\| \sqrt{\textbf{c}_{0}}\,z_{1}\right\|^{2}+\kappa^{2}\left\| z_{1x}\right\|^{2}+\left\| \sqrt{\textbf{p}(x)}\,z_{1xx}\right\|^{2}+\left\| z_{2}\right\|^{2}+(l+\textbf{c}_{1})\left\| z_{2x}\right\|^{2}\\
	\end{equation*}
	where $z=(z_{1},z_{2})\in H_{0}^{2}\times H_{0}^{1} $. It is not hard to see that ${\mathcal J}$ is linear, bounded and coercive. Therefore, by  Lax-Milgram's Theorem exists a unique solution $z=(z_{1},z_{2})=(u\,,\theta)\in H_{0}^{2}\times H_{0}^{1}$ for the problem \eqref{125}--\eqref{126}.\\
	
	\noindent	
	On other hand, from\eqref{119} we have that $v\in H_{0}^{2}$ and by  \eqref{123} is not hard to see that $\eta(0)=0$ and
	\begin{equation*}
	\int_{0}^{\infty}\mu(s)\left\| \eta_{x}\right\|^{2} \ ds \leq 2\,\mu_{0}\left\|\theta_{x} \right\|^{2}+ 2\left\|\eta^{*} \right\|_{{\cal M}}^{2}
	\end{equation*}
	Therefore $\eta\in{\cal M}$.
	\noindent	
	Additionally from \eqref{120}  we have
	\begin{equation*}
	\left( \textbf{p}(x)\,u_{xx}\right) _{xx}-\kappa^{2}\,u_{xx}\in L^{2}
	\end{equation*}
	thus, from the regularity theory for the linear elliptic equations we obtain
	\begin{equation*}
	\textbf{p}(x)\,u_{xx}-\kappa^{2}\,u\in H^{2}
	\end{equation*}
	and then $\Phi\in \textbf{D}(\mathbb{A})$ solve the problem
	\begin{equation*}
	\left( I-\mathbb{A}\right)\,\Phi=\Phi^{*}
	\end{equation*}
	therefore $ I-\mathbb{A}$ is a surjective operator. As a consequence of Lumer-Phillips theorem  $\mathbb{A}$ is the infinitesimal generator of a $C_{0}$-semigroup $ \textbf{B}(t)$ of contractions on the Hilbert space ${\cal H}$.
\end{proof}

From this theorem follows well-posedness for the abstract Cauchy problem \eqref{115} thanks to the semigroup theory of linear operators. In particular, the following theorem \cite{pazy} is obtained immediately.

\begin{theorem}
	For any $\Phi_{0}\in \textbf{D}(\mathbb{A})$. Then $\Phi(t)$ is a strong solution of \eqref{115} satisfying
	\begin{equation*}
	\Phi\in C\left( (0,\infty);\textbf{D}(\mathbb{A})\right)\cap C^{1}\left( (0,\infty);\mathcal{H}\right)
	\end{equation*}
\end{theorem}
\section*{Exponential stability}
\label{sec:3}
In this section, we focus on to prove the uniform exponential stability for  $\textbf{B}(t)=e^{\mathbb{A}\,t} $  semigroup associated to the operator $\mathbb{A}$ , given in the section 2. For this we will see that the energy of the system decays uniformly exponentially with time. The achievement of this objective will be based on the use of multiplicative techniques, which essentially consist of the following lemmas

\begin{lemma}\label{lem0}
	For every solution $\Phi(t)$ of the system \eqref{107}--\eqref{111} the total energy 	$\mathcal{E}:\mathbb{R}^{+}\rightarrow\mathbb{R}^{+}$ is given at time $t$ by
	\begin{equation}
	2\,{\cal{E}} = \int_{\Omega} \textbf{p}(x)\,|u_{xx}|^{2}\ dx+\kappa^{2}\int_{\Omega}\,|u_{x}|^{2}\ dx+ \int_{\Omega}|v|^{2} \ dx+ \int_{\Omega}|\theta|^{2} \ dx+\int_{0}^{+\infty}\mu(s)\Arrowvert\eta_{x}\Arrowvert^{2}\ ds
	\end{equation}
	and satisfies
	\begin{equation}
	\frac{d}{dt}\,{\cal{E}}(t)=-2 \int_{\Omega}\textbf{g}(x)|v|^{2}\ dx-l \int_{\Omega}|\theta_{x}|^{2}\ dx+\frac{1}{2}\int_{0}^{+\infty}\mu'(s)\Arrowvert\eta_{x}\Arrowvert^{2}\ ds \\
	\end{equation}	
\end{lemma}	
\begin{proof}
	By multiplying the equations \eqref{107} and \eqref{108} by $v$ and $\theta$ , boundary conditions \eqref{110} and integrating by parts we follows the result.
\end{proof}	

\begin{lemma}\label{lem1}
	Let $\Phi(t)$ be solution of the problem \eqref{107}--\eqref{111}. Then the time derivative of the functional $F_{1}$, defined by
	\begin{equation}\label{128}
	F_{1}(t)=\int_{\Omega}u\,v\ dx+\int_{\Omega}\textbf{g}(x)\,u^{2}\ dx
	\end{equation}	
	satisfies the inequality
	\begin{equation}\label{129*}
	\frac{d}{dt}F_{1}(t) \leq  -\int_{\Omega}\textbf{p}(x)\,|u_{xx}|^{2}\ dx-\frac{\kappa^{2}}{4} \int_{\Omega}|u_{x}|^{2}\ dx+C_{\kappa}\int_{\Omega}|v|^{2}\ dx+\frac{\beta^{2}}{2\,\kappa^{2}}\int_{\Omega}|\theta|^{2}\ dx
	\end{equation}
	where $C_{\kappa}\geq 5$
\end{lemma}
\begin{proof}
	Differentiating \eqref{128} in $t$- variable, we follows:
	\begin{equation}\label{129}
	\frac{d}{dt}F_{1}(t)=\int_{\Omega}u\,v_{t}\ dx+\int_{\Omega}v^{2}\ dx+\,\frac{d}{dt}\int_{\Omega}\textbf{g}(x)\,u^{2}\ dx
	\end{equation}	
	Using 	\eqref{107} in \eqref{129} yields
	\begin{align}
	\frac{d}{dt}F_{1}(t) & =\int_{\Omega}u\,\left[ -\left( \, \textbf{p}(x)\,u_{xx}\,\right)_{xx} -2\,g(x)\,v-2\,\kappa\, v_{x}+\kappa^{2}\, u_{xx}-\beta\,\theta_{x} \right] \ dx \nonumber \\
	&+\int_{\Omega}v^{2}\ dx+\,\frac{d}{dt}\int_{\Omega}\textbf{g}(x)\,u^{2}\ dx \label{130}
	\end{align}
	Integrating by parts \eqref{130}, simplifying, follows
	\begin{align}
	\frac{d}{dt}F_{1}(t) =& -\int_{\Omega}\textbf{p}(x)\,u_{xx}^{2}\ dx+2\,\kappa\int_{\Omega}\,u_{x}\,v\ dx-\kappa^{2}\int_{\Omega}\,u_{x}^{2}\ dx+\beta\int_{\Omega}u_{x}\,\theta\ dx\nonumber\\
	&+\int_{\Omega}v^{2}\ dx \label{131}
	\end{align}
	by carefully using Young's inequality in \eqref{131} we get
	\begin{align*}
	\frac{d}{dt}F_{1}(t) =& -\int_{\Omega}\textbf{p}(x)\,|u_{xx}|^{2}\ dx+\frac{\kappa^{2}}{4}\int_{\Omega}|u_{x}|^{2}\ dx+4\int_{\Omega}|v|^{2}\ dx-\kappa^{2}\int_{\Omega}\,|u_{x}|^{2}\ dx\nonumber\\
	&+\frac{\kappa^{2}}{2}\int_{\Omega}|u_{x}|^{2}\ dx
	+\frac{\beta^{2}}{2\,\kappa^{2}}\int_{\Omega}|\theta|^{2}\ dx+\int_{\Omega}|v|^{2}\ dx 
	\end{align*}
	therefore	
	\begin{equation*}
	\frac{d}{dt}F_{1}(t) \leq  -\int_{\Omega}\textbf{p}(x)\,|u_{xx}|^{2}\ dx-\frac{\kappa^{2}}{4} \int_{\Omega}|u_{x}|^{2}\ dx+C_{\kappa}\int_{\Omega}|v|^{2}\ dx+\frac{\beta^{2}}{2\,\kappa^{2}}\int_{\Omega}|\theta|^{2}\ dx
	\end{equation*}
	where $C_{\kappa}\geq 5$, hence lemma is follows.
\end{proof}

\begin{lemma}\label{lem3}
	Let $\Phi(t)$ be solution of the problem \eqref{107}--\eqref{111}. Then the functional $I$, defined by
	\begin{equation}\label{133}
	I(t)=-\int_{0}^{\infty}\mu(s)\left( \int_{\Omega}\theta_{t}\,\eta\ dx\right)\ ds
	\end{equation}	
	satisfies the inequality
	\begin{equation}\label{133*}
	I(t)\leq C_{1}\int_{\Omega}|v|^{2}\ dx+ C_{2}\int_{\Omega}|\theta_{x}|^{2}\ dx -C_{3}\int_{0}^{\infty}\mu'(s) \left\| \eta_{x}\right\|^{2}  \ ds
	\end{equation}
	for positive constants:
	\begin{equation*}
	C_{1}=\frac{\beta\,\mu_{0}\,\sigma_{2}}{2},\quad C_{2}=\frac{l\,\mu_{0}\,\sigma_{1}}{2},\quad C_{3}=\frac{l}{2\,\sigma_{1}\,\delta_{1}}+\frac{\mu_{0}\,\beta}{2\,\sigma_{2}\,\delta_{1}}+\frac{l\,\mu_{0}}{\delta_{1}}
	\end{equation*}		
\end{lemma}
\begin{proof}
	In fact, by using \eqref{108} in \eqref{133} follows
	\begin{align}
	I(t)&=-\int_{0}^{\infty}\mu(s)\left( \int_{\Omega}\left[l\,\theta_{xx}+ \int^{\infty}_{0}\mu(s)\,\eta_{xx}(s) \ ds-\beta\,v_{x}\right]\,\eta\ dx\right)\ ds \nonumber\\
	&=-l\int_{0}^{\infty}\mu(s) \int_{\Omega}\theta_{xx}\,\eta\ dx\ ds+\beta\int_{0}^{\infty}\mu(s) \int_{\Omega}v_{x}\,\eta\ dx\ ds\nonumber\\
	&-l\int_{\Omega}\left|\int_{0}^{\infty}\mu(s)\, \eta_{x}\ ds\right|^{2} \ dx \label{134}
	\end{align}
	Integrating by parts and using $\mathbb{H}_{4}$ and  Young's inequality in \eqref{134} follows
	\begin{align*}
	I(t)&\leq \frac{l\,\mu_{0}\,\sigma_{1}}{2}\int_{\Omega}|\theta_{x}|^{2}\ dx-\frac{l}{2\,\sigma_{1}\,\delta_{1}}\int_{0}^{\infty}\mu'(s) \left\| \eta_{x}\right\|^{2}  \ ds+\frac{\beta\,\mu_{0}\,\sigma_{2}}{2}\int_{\Omega}|v|^{2}\ dx \nonumber\\
	&-\frac{\mu_{0}\,\beta}{2\,\sigma_{2}\,\delta_{1}}\int_{0}^{\infty}\mu'(s) \left\| \eta_{x}\right\|^{2}  \ ds -\frac{l\,\mu_{0}}{\delta_{1}}\int_{0}^{\infty}\mu'(s) \left\| \eta_{x}\right\|^{2}  \ ds \quad\mbox{for}\quad \sigma_{1}\,,\sigma_{2}>0
	\end{align*}
	Rewriting the previous inequality
	\begin{align*}
	I(t)\leq C_{1}\int_{\Omega}|v|^{2}\ dx+ C_{2}\int_{\Omega}|\theta_{x}|^{2}\ dx -C_{3}\int_{0}^{\infty}\mu'(s) \left\| \eta_{x}\right\|^{2}  \ ds
	\end{align*}
	where
	\begin{equation*}
	C_{1}=\frac{\beta\,\mu_{0}\,\sigma_{2}}{2},\quad C_{2}=\frac{l\,\mu_{0}\,\sigma_{1}}{2},\quad C_{3}=\frac{l}{2\,\sigma_{1}\,\delta_{1}}+\frac{\mu_{0}\,\beta}{2\,\sigma_{2}\,\delta_{1}}+\frac{l\,\mu_{0}}{\delta_{1}}
	\end{equation*}
	hence lemma has been proved.
\end{proof}

\begin{lemma}\label{lem2}
	Let $\Phi(t)$ be solution of the problem \eqref{107}--\eqref{111}. Then the time derivative of the functional $F_{2}$, defined by
	\begin{equation}\label{132}
	F_{2}(t)=-\int_{0}^{\infty}\mu(s)\left( \int_{\Omega}\theta\,\eta\ dx\right)\ ds
	\end{equation}	
	satisfies the inequality
	\begin{align}
	\frac{d}{dt}F_{2}(t) &\leq C_{1}\int_{\Omega}|v|^{2}\ dx+ C_{2}\int_{\Omega}|\theta_{x}|^{2}\ dx+\left( \frac{\sigma_{3}}{2}-\mu_{0}\right) \int_{\Omega}|\theta|^{2} \ dx \nonumber\\
	&-\left(C_{3} +\frac{C_{p}}{2\,\sigma_{3}}\right) \int_{0}^{\infty}\mu'(s) \left\| \eta_{x}\right\|^{2}  \ ds \label{134*}
	\end{align}	
\end{lemma}
\begin{proof}
	Differentiating \eqref{132} in $t$-variable and using \eqref{109} we have:
	\begin{equation}\label{135}
	\frac{d}{dt}F_{2}(t) =-\int_{0}^{\infty}\mu(s)\left( \int_{\Omega}\theta_{t}\,\eta\ dx\right)\ ds-\mu_{0}\int_{\Omega}|\theta|^{2} \ dx+\int_{0}^{\infty}\mu(s)\left( \int_{\Omega}\theta\,\eta_{s}\ dx\right)\ ds
	\end{equation}
	Integrating by parts and using the Young and Poincar\'e inequalities in the last integral of \eqref{135} we get
	\begin{equation}\label{136}
	\int_{0}^{\infty}\mu(s)\left( \int_{\Omega}\theta\,\eta_{s}\ dx\right)\ ds\leq\frac{\sigma_{3}}{2}\int_{\Omega}|\theta|^{2} \ dx-\frac{C_{p}}{2\,\sigma_{3}}\int_{0}^{\infty}\mu'(s) \left\| \eta_{x}\right\|^{2}  \ ds\quad\mbox{for}\quad \sigma_{3}>0
	\end{equation}
	where $C_{p}$ is the Poincare's constant. On other hand, by using \eqref{133*} and \eqref{136} in \eqref{135}  follows
	\begin{align*}
	\frac{d}{dt}F_{2}(t) &\leq C_{1}\int_{\Omega}|v|^{2}\ dx+ C_{2}\int_{\Omega}|\theta_{x}|^{2}\ dx+\left( \frac{\sigma_{3}}{2}-\mu_{0}\right) \int_{\Omega}|\theta|^{2} \ dx \nonumber\\
	&-\left(C_{3} +\frac{C_{p}}{2\,\sigma_{3}}\right) \int_{0}^{\infty}\mu'(s) \left\| \eta_{x}\right\|^{2}  \ ds
	\end{align*}
	completing our proof.
\end{proof}

Now we are able to define the following Lyapunov functional
\begin{equation*}
\mathcal{L}(t)=N\mathcal{E}(t)+N_{1} F_{1}+N_{2} F_{2}
\end{equation*}
where $N, N_{1},N_{2}$ are positive constants that will be chosen later carefully. On the other hand, to ensure the equivalence of the functional $ \mathcal{E}(t) $ and $ \mathcal{L}(t) $, we state the following lemma.
\begin{lemma}\label{lem4}
	For $ N $ large enough there are two positive constants $\gamma_{1}$, $\gamma_{2}$ so that
	\begin{equation*}
	\gamma_{1}\,\mathcal{E}(t)\leq\mathcal{L}(t)\leq\gamma_{2}\,\mathcal{E}(t)
	\end{equation*}
\end{lemma}
\begin{proof}
	According to the  Young and Poincar\'e inequalities, we have:
	\begin{align*}
	&F_{1}(t)\leq\frac{\sigma\,C_{p}}{\kappa^{2}}\int_{\Omega}|u_{x}|^{2}\ dx+\frac{1}{2\,\sigma}\int_{\Omega}|v|^{2}\ dx+2\,\alpha_{3}\,C_{p}\int_{\Omega}|u_{x}|^{2}\ dx\leq \zeta_{1}\, \mathcal{E}(t)\quad\mbox{for}\quad\sigma>0\\
	&F_{2}(t)\leq\frac{\mu_{0}}{2}\int_{\Omega}|\theta|^{2}\ dx+\frac{1}{2}\left\| \eta\right\|_{{\cal M}}^{2}\leq \zeta_{2}\, \mathcal{E}(t)
	\end{align*}
	where $\zeta_{1}=max\left\lbrace \frac{\sigma\, C_{p}}{\kappa^{2}}+2\,\alpha_{3}\,C_{p}\,,\frac{1}{2\,\sigma}\right\rbrace $ and  $\zeta_{2}=max\left\lbrace\frac{\mu_{0}}{2}\,,\frac{1}{2} \right\rbrace $, therefore there is a positive constant $\gamma_{0}$ that
	\begin{equation*}
	\left|\mathcal{L}(t)-N\mathcal{E}(t)\right|=\left|N_{1} F_{1}+N_{2} F_{2} \right| \leq \gamma_{0}\, \mathcal{E}(t)
	\end{equation*}	
	Therefore, we can choose $\gamma_{1}=N-\gamma_{0}$ and $\gamma_{2}=N+\gamma_{0}$ as long as $N-\gamma_{0}>0$. In this way the lemma is proven
\end{proof}

We state our main results as follows.
\begin{theorem}
	Assume that $\mu$ satisfies the conditions $\mathbb{H}_{1}-\mathbb{H}_{4}$. Then there exist positive constants $K$ and $\gamma$ such that the relation
	\begin{equation*}
	{\mathcal E}(t) \leq K\,{\mathcal E}(0)\,e^{-\,\gamma\,t}
	\end{equation*}
	holds for every $t\geq 0$
\end{theorem}
\begin{proof}
	Taking time derivative of $\mathcal{L}$ we obtain
	\begin{equation*}
	\mathcal{L'}(t)=N\,\mathcal{E'}(t)+N_{1}\, F'_{1}+N_{2}\, F'_{3}
	\end{equation*}
	according to the lemma \ref{lem1} and \ref{lem2} we get
	
	\begin{align}
	\mathcal{L'}(t)&\leq N\left[ -2 \int_{\Omega}\textbf{g}(x)|v|^{2}\ dx-l \int_{\Omega}|\theta_{x}|^{2}\ dx +\frac{1}{2}\int_{0}^{+\infty}\mu'(s)\lVert\,\eta_{x}\rVert^{2}\ ds \ dx\right] \nonumber\\
	& +N_{1}\left[-\int_{\Omega}\textbf{p}(x)\,|u_{xx}|^{2}\ dx-\frac{\kappa^{2}}{4} \int_{\Omega}|u_{x}|^{2}\ dx+C_{\kappa}\int_{\Omega}|v|^{2}\ dx+\frac{\beta^{2}}{2\,\kappa^{2}}\int_{\Omega}|\theta|^{2}\ dx  \right]\nonumber \\
	&+N_{2}\left[ C_{1}\int_{\Omega}|v|^{2}\ dx+ C_{2}\int_{\Omega}|\theta_{x}|^{2}\ dx+\left( \frac{\sigma_{3}}{2}-\mu_{0}\right) \int_{\Omega}|\theta|^{2} \ dx\right]  \nonumber\\
	&-N_{2}\left[\left(C_{3} +\frac{C_{p}}{2\,\sigma_{3}}\right) \int_{0}^{\infty}\mu'(s) \left\| \eta_{x}\right\|^{2}  \ ds\right]\label{309}
	\end{align}
	By grouping terms in \eqref{309} and using \eqref{105} we get
	\begin{align*}
	\mathcal{L'}(t)& \leq -\left[  2\,\alpha_{4}\,N-C_{\kappa}\,N_{1}-C_{1}\,N_{2}\right] \int_{\Omega}|v|^{2}\ dx -\left[ Nl-N_{2}\,C_{2}\right] \int_{\Omega}|\theta_{x}|^{2}\ dx\nonumber\\
	&+\left[  \frac{N}{2}-N_{2}\left(C_{3} +\frac{C_{p}}{2\,\sigma_{3}}\right)\right]\int_{0}^{\infty}\mu'(s) \left\| \eta_{x}\right\|^{2}  \ ds -N_{1}\int_{\Omega}\textbf{p}(x)\,u_{xx}^{2}\ dx\nonumber\\
	&-\left[N_{2}\left(\mu_{0}- \frac{\sigma_{3}}{2}\right) -N_{1}\frac{\beta^{2}}{2\,\kappa^{2}} \right]\int_{\Omega}|\theta|^{2} \ dx -N_{1}\,\frac{\kappa^{2}}{4}\int_{\Omega}|u_{x}|^{2}\ dx
	\end{align*}
	this last inequality can be rewritten as
	\begin{align}
	\mathcal{L'}(t) \leq & - \mathcal{C}_{1}\int_{\Omega}|v|^{2}\ dx -\mathcal{C}_{2} \int_{\Omega}|\theta_{x}|^{2}\ dx+\mathcal{C}_{3}\int_{0}^{\infty}\mu'(s) \left\| \eta_{x}\right\|^{2}  \ ds -N_{1}\int_{\Omega}\textbf{p}(x)\,u_{xx}^{2}\ dx\nonumber\\
	&-\mathcal{C}_{4}\int_{\Omega}|\theta|^{2} \ dx -N_{1}\,\frac{\kappa^{2}}{4}\int_{\Omega}|u_{x}|^{2}\ dx\label{310}
	\end{align}
	where
	\begin{align*}
	\mathcal{C}_{1}=&  2\,\alpha_{4}\,N-C_{\kappa}\,N_{1}-C_{1}\,N_{2},
	\quad\mathcal{C}_{2}= Nl-N_{2}\,C_{2}\\
	\mathcal{C}_{3}=&\frac{N}{2}-N_{2}\left(C_{3} +\frac{C_{p}}{2\,\sigma_{3}}\right),
	\quad\mathcal{C}_{4}=N_{2}\left(\mu_{0}- \frac{\sigma_{3}}{2}\right) -N_{1}\frac{\beta^{2}}{2\,\kappa^{2}}
	\end{align*}
	then the constants $ \mathcal{C}_{i}, i = 1,2,3,4 $ are strictly positive if and only if
	\begin{equation*}
	\mu_{0}>\frac{\sigma_{3}}{2},\quad N_{1}>0,\quad  N_{2}>\frac{N_{1}\,\beta^{2}}{2\,\kappa^{2}\left(\mu_{0}-\frac{\sigma_{3}}{2} \right)}
	\end{equation*}
	and $N$ is chosen as
	\begin{equation*}
	N>max\left\lbrace \frac{C_{\kappa}N_{1}+C_{1}N_{2}}{2\alpha_{4}},\frac{N_{2}C_{2}}{l},2N_{2}\left(C_{3} +\frac{C_{p}}{2\,\sigma_{3}}\right)\right\rbrace
	\end{equation*}
	therefore for a certain positive constant $ \lambda $, the inequality  \eqref {310} becomes
	\begin{equation} \label{311}
	\mathcal{L'}(t) \leq  -\lambda\,\mathcal{E}(t)+\mathcal{C}_{3}\int_{0}^{\infty}\mu'(s) \left\| \eta_{x}\right\|^{2}  \ ds\leq  -\lambda\,\mathcal{E}(t)
	\end{equation}
	according to lemma \ref{lem4}, the inequality \eqref{311} becomes
	\begin{equation*}
	\mathcal{L'}(t) \leq-\gamma\,\mathcal{L}(t) ,\quad\forall t>0
	\end{equation*}
	where $\gamma=\frac{\lambda}{\gamma_{2}}$, then by integrating we get
	\begin{equation*}
	\mathcal{L'}(t) \leq K\,\mathcal{L}(0)\,e^{-\gamma\,t} ,\quad\forall t>0
	\end{equation*}
	or equivalently 
	\begin{equation*}
	\mathcal{E}(t)\leq K\,\mathcal{E}(0)e^{-\gamma\,t},\quad\forall t>0
	\end{equation*}	
	In this way our main theorem is proven.
	
\end{proof}

\vskip 0,65 true cm

\section*{Conclusion} 
 When we proof the well-posedness and the exponential decay (via the energy method) and not another weaker type of decay as for example the polynomial decay, we have made an improvement of the works named in the section 1 offering with it a much more realistic model from the physical point of view. Let us say that these results can be
 improved furthermore by considering the Cattaneo's law \cite{cat} and Gurtin-Pipkin law \cite{gur} instead of equation (1.4), obtaining an effective prediction for heating propagation  on the structure. All these statements constitute a promising set of new  questions to be addressed in further research.

\section*{Acknowledgements}

\end{document}